\documentclass[10pt,oneside,reqno]{amsart}
\usepackage{amsthm}
\usepackage{graphicx}

\usepackage{amsmath, amssymb, amsthm,color}
\usepackage{amssymb,longtable}
\usepackage{amsfonts}
\usepackage{hyperref} 

\theoremstyle{definition}
    \newtheorem{thm}{Theorem}[section]
    \newtheorem{lem}{Lemma}[section]
    \newtheorem{rem}{Remark}[section]
    \newtheorem{defi}{Definition}[section]
\renewcommand{\div}{{\rm div}}
\newcommand{\curl}{{\rm curl}} 
\newcommand{\dq}{{\Delta_q}}
\newcommand{\dk}{{\Delta_k}}
\newcommand{\dH}{{\dot{H}}}
\newcommand{\tL}{{\tilde{L}}}
\renewcommand{\-}{{v_-}}
\newcommand{\+}{{v_+}}
\renewcommand{\v}{{v_\pm^i}}
\begin{document}
\setlength{\parindent}{0em}
\title{Global well posedness for a two-fluid model}

\author{Yoshikazu Giga$ ^1 $}
\address{1. Graduate School of Mathematical Sciences\\
University of Tokyo\\  3-8-1 Komaba Meguro-ku\\   Tokyo 153-8914 \\ Japan}
\email{labgiga@ms.u-tokyo.ac.jp}

\author{Slim Ibrahim$ ^{2,3} $}
\address{2. Department of Mathematics and Statistics, University of Victoria\\
	PO Box 3060 STN CSC\\   Victoria, BC, V8P 5C3\\ Canada}
\email{ibrahim@math.uvic.ca, shengyis@uvic.ca}
\urladdr{http://www.math.uvic.ca/~ibrahim/}

\author{Shengyi Shen$ ^{2,3} $}
\address{3. Pacific Institute for the Mathematical Sciences and Department of Mathematics and Statistics\\ University of Victoria\\
	PO Box 3060 STN CSC\\   Victoria, BC, V8P 5C3\\ Canada}
\email{ibrahim@math.uvic.ca, shengyis@uvic.ca}

\author{Tsuyoshi Yoneda$ ^4 $}
\address{4. Graduate School of Mathematical Sciences\\ University of Tokyo\\
 3-8-1 Komaba Meguro-Ku\\Tokyo 153-8914\\ Japan}
\email{yoneda@ms.u-tokyo.ac.jp}

\subjclass[2010]{76W05, 76N10, 35Q30.}
\keywords{Navier-Stokes equations, Maxwell equations, NSM, energy decay, local existence, global existence, long time existence.}


\begin{abstract}
We study a two fluid system which models the motion of a charged fluid with Rayleigh friction, and in the presence of an electro-magnetic field satisfying Maxwell's equations. We study the well-posdness of the system in both space dimension{}{s} two and three. Regardless of the size of the initial data, we first prove the global well-posedness of the Cauchy problem when the space dimension is two. However, in space dimension three, we construct global weak-solutions {\it \`a la Leray}, and we prove the local well-posedness of Kato-type solutions. These solutions turn out to be global when the initial data are sufficiently small. Our results extend Giga-Yoshida {}{(1984)} \cite{YZ} ones to the space dimension two, and improve them in terms of requiring less regularity on the velocity fields. 
\end{abstract}
\date{}
\maketitle

\bibliographystyle{plain}

\section{Introduction}
We consider the following two-fluid incompressible Navier-Stokes-Maxwell system (NSM):

\begin{equation}\label{physic_model}
\begin{cases}
nm_-\partial_t v_-=\nu_-\Delta v_- - nm_- (v_-\cdot \nabla) v_- - en(E+v_-\times B) - R- \nabla p_- \\
nm_+\partial_t v_+=\nu_+\Delta v_+ - nm_+ (v_+\cdot \nabla) v_+ + eZn(E+v_+\times B) + R- \nabla p_+ \\
\partial_t E = \frac{1}{\varepsilon_0\mu_0}\nabla\times B-\frac{ne}{\varepsilon_0}(Zv_+-v_-) \\
\partial_t B = -\nabla\times E \\
R:=-\alpha(v_+-v_-)\\
\text{div} v_-=\text{div} v_+=\text{div} B=\text{div} E=0\\
\end{cases}
\end{equation}
with the initial data
\begin{equation*}
v_-|_{t=0}=v_{-,0},\quad v_+|_{t=0}=v_{+,0},\quad B|_{t=0}=B_0,\ E|_{t=0}=E_0.
\end{equation*}
The system models the motion of a plasma of cations (positively charged) and anions (negatively charged)
particles with approximately equal masses $m_\pm$. The constants $n$ stands for the number density, and $e$ is the
elementary charge. The charge number is given by $Z$ and $\varepsilon_0$ represents the vacuum
dielectric constant and $\mu_0$ is the vacuum permeability. The term $R$ comes from Rayleigh friction with a constant coefficient $\alpha>0$ allowing a momentum transfer between the two components of the fluid. \\
The vector fields $v_-$ and $v_+:$ $\mathbb{R}^+_t\times \mathbb{R}^d_x\to \mathbb{R}^3$ represent then
the velocities of the anions and cations, respectively.
The electromagnetic field is represented by $E$, $B$: $\mathbb{R}^+_t\times \mathbb{R}^d_x\to \mathbb{R}^3$. Here the space
dimension is either
$d=3$ or 2, 
and $\nu_\pm$ is the kinematic viscosity of the two-fluid and the scalar function $p^\pm$ stands
for its pressure. We refer to \cite{YZ} and the references therein for more details about the model.
The third equation is the Amp\`ere-Maxwell equation for an electric field $E$. The equations on the velocities are the momentum equation,
and the fourth equation is nothing but Farady's law.
For a detailed introduction to the NSM, we refer to Davidson \cite{B} and Biskamp \cite{D}.\\

If we assume all the physical constants are one, then the system has nice cancellations (see Remark \ref{cancel}). However, the original physical model \eqref{physic_model} does not have such nice cancellation property.
Our system is a coupling between a parabolic and a hyperbolic equations and we cannot hope to gain regularity from the Maxwell equations. 
In this paper, we mainly study the existence and uniqueness of global-in-time solution to \eqref{physic_model} in 2D. 
Note that we can also show existence of the global solution even in the case $\alpha=0$, since we treat finitely but arbitrary long 
time intervals. 
While in 3D, we construct a global weak solution {\it \` a la Leray} and a local-in-time strong solution for system \eqref{physic_model}. We also show that this solution becomes global if its initial data is sufficiently small. 
{}{Mathematical analysis on this problem went back to the work of
Giga-Yoshida \cite{YZ}. They considered the system in a three-dimensional
bounded domain with no-slip and perfectly conductive boundary condition
and prove (unique) local solvability as well as global-in-time
solvability for a small initial data whose magnetic effect is small
compared with velocity. Their method is based on nonlinear semigroup theory
	initiated by K\=omura \cite{Ko} which applied to the Navier-Stokes system \cite{YG}.
Our results extend those \cite{YZ} to the space dimension two, and improve them in terms of requiring less regularity on the velocity fields. 
In the 2D case, we basically use the classical compactness argument (c.f. \cite{M}, \cite{L}) to prove our result. For the 3D case, the proof for existence of global weak solutions goes along the the same lines as for the incompressible Navier-Stokes equations. For the sake of completeness, we outline it in this paper. Finally, we also emphasize that our results are in striking difference with the following 
slightly modified Navier-Stokes-Maxwell one fluid model studied in \cite{IK}, \cite{IY} and \cite{GIM}.

\begin{eqnarray}
\label{NS1}
\ \left\{
\begin{array}{rclll}
\frac{\partial v}{\partial t} +   v\cdot\nabla  v - \nu
\Delta v +\nabla p & =&     j \times   B 
\\
\partial_t E  - \curl\, B&=&-j \\
\partial_tB  +  \curl\, E
& = & 0&
\\
\div v=\div B &=&0
\\
\sigma( E + v {{\times}}    B) &=&j.
\end{array}
\right.
\end{eqnarray}
It is important to mention that the existence of global weak 
solutions of \eqref{NS1} is still an outstanding important open problem in space dimension three. {}{The global well-posedness in 2D is treated in \cite{M}.}
The local well-posedness and the existence of global small solutions were studied in \cite{IK} and \cite{GIM} for initial data in $u_0, E_0, B_0 \in 
{\dot H^{\frac12}\times\dot H^{\frac12}\times \dot H^{\frac12}}$. 
We also construct local-in-time mild-type solution for the 3D case.
The proof combines {\it a priori} estimate techniques with the Banach fixed point theorem. 

We use the short-hand notations throughout the paper $ L^p_T X=L^p(0,T; X) $,
and we also use the notation $A\lesssim B$ which means $A\leq CB$, where $C>0$ is a universal constant. Also we define the weak solution of our system:
\begin{defi}
	A time-dependent vector field $ (v_-,v_+,E,B) $ with components in $ L^2_{loc}((0,T]\times \mathbb{R}^{{}{d}}) $ is a weak solution to \eqref{physic_model} if for any $ t<T $ and any smooth, compactly supported, divergence-free {}{test function} $ \phi(t,x) $, the vector field $ (v_-,v_+,E,B) $ solves 
	\begin{equation*}
	\begin{cases}
	\quad nm_-\int_{{\mathbb{R}^{{}{d}}}}( v_-\cdot \phi)(t,x) -( v_-\cdot \phi)(0,x)dx-nm_-\int_0^t\int_{\mathbb{R}^{{}{d}}}( v_-\cdot \partial_t\phi)(t',x)dxdt' \\
	=\int_0^t\int_{\mathbb{R}^{{}{d}}}\nu_- v_-\cdot\Delta\phi + nm_- v_-\otimes v_-:\nabla\phi - en(E+v_-\times B)\cdot\phi - R\cdot\phi dxdt'{}{,} \\
	\\
	\quad nm_+\int_{{\mathbb{R}^{{}{d}}}}( v_+\cdot \phi)(t,x) -( v_+\cdot \phi)(0,x)dx-nm_+\int_0^t\int_{\mathbb{R}^{{}{d}}}( v_+\cdot \partial_t\phi)(t',x)dxdt'\\
	=\int_0^t\int_{\mathbb{R}^{{}{d}}}\nu_+ v_+\cdot\Delta\phi + nm_+ v_+\otimes v+:\nabla\phi + eZn(E+v_+\times B)\cdot\phi + R\cdot\phi dxdt'{}{,} \\
	\\
	\quad \int_{{\mathbb{R}^{{}{d}}}}( E\cdot \phi)(t,x) -( E\cdot \phi)(0,x)dx-\int_0^t\int_{\mathbb{R}^{{}{d}}}(E\cdot \partial_t\phi)(t',x)dxdt' \\
	= \int_0^t\int_{\mathbb{R}^{{}{d}}}\frac{1}{\varepsilon_0\mu_0} B\cdot(\nabla\times\phi)-\frac{ne}{\varepsilon_0}(Zv_+-v_-)\cdot\phi dxdt'{}{,} \\
	\\
	\quad \int_{{\mathbb{R}^{{}{d}}}}( B\cdot \phi)(t,x) -( B\cdot \phi)(0,x)dx-\int_0^t\int_{\mathbb{R}^{{}{d}}}(B\cdot \partial_t\phi)(t',x)dxdt'\\
	= \int_0^t\int_{\mathbb{R}^{{}{d}}}-E\cdot(\nabla\times\phi) dxdt'{}{,}\\
	\\
	R:=-\alpha(v_+-v_-){}{.}
	\end{cases}
	\end{equation*}
\end{defi}
The following is our first main result.
\begin{thm}[Global well-posedness for 2D]\label{2dMain}
		Assuming $ \+(t=0), \-(t=0) \in L^2 $ and $ E(t=0), B(t=0) \in {L^2} $. Then for any $ 0<s_1'<1 $ and any $ T>0 $, there exists a unique {}{weak} solution of  \eqref{physic_model} such that $ \+,\-\in L^1(0,T;H^{s_1'+1})\cap C([0,T];L^2)$, $ E,B\in C([0,T];L^2) $. Furthermore, the solution satisfies the following estimate, 
			\[
			{\|E\|_{C([0,T];L^2)} + \|B\|_{C([0,T];L^2)} \lesssim C_0C_T}
			\]
			and 
			\[
			\|v\|_{L^1_T H^{{}{s_1'+1}}} \lesssim C_0C_T^2,
			\]
			where $ C_0=\|\-\|_{L^2}+\|\+\|_{L^2}+\|E_0\|_{L^2}+\|B_0\|_{L^2} $, $ C_T=C\max(1,T) $ with $ C $ a universal constant, and $ v=v_\pm $.
\end{thm}
Our second result concerns the existence of global weak solution to \eqref{physic_model} in 3D case. The definition of weak solution is the following.
\begin{thm}\label{3dMain1}
	For initial data $ \+(t=0), \-(t=0) \in L^2 $ and $ E(t=0), B(t=0) \in {L^2} $ with $\text{div}\ v_{-,0}=\text{div}\ v_{+,0}=0$,
	there exists a weak solution
	\begin{equation*}
	\begin{cases}
	v_-\in {L^\infty}(0,\infty;L^2)\cap L^2(0,\infty;\dot H^1) \\
	v_+\in {L^\infty}(0,\infty;L^2)\cap L^2(0,\infty;\dot H^1)\\
	B\in {}{C}([0,\infty);L^2)\\
	E\in {}{C}([0,\infty);L^2)\\
	R:=-\alpha(\+-\-)\in L^2(0,\infty;L^2),
	\end{cases}
	\end{equation*}
satisfying the following energy inequality
	\begin{eqnarray}\label{energyInequality}
	\frac{nm_-}{2\varepsilon_0}\|v_-\|^2_{L^2} + \frac{nm_+}{2\varepsilon_0}\|v_+\|^2_{L^2} + \frac{1}{2} \|E\|^2_{L^2}+\frac{1}{2\varepsilon_0\mu_0} \|B\|^2_{L^2} && \\
	+\frac{\nu_-}{\varepsilon_0}\|v_-\|^2_{L^2_t\dot{H}^1} + \frac{\nu_+}{\varepsilon_0}\|v_+\|^2_{L^2_t\dot{H}^1} + \frac{\alpha}{\varepsilon_0}\|v_--v_+\|^2_{L^2_tL^2} &&\nonumber \\
	\le&& \nonumber\\
	\frac{nm_-}{2\varepsilon_0}\|v_-(0)\|^2_{L^2} + \frac{nm_+}{2\varepsilon_0}\|v_+(0)\|^2_{L^2} + \frac{1}{2} \|E(0)\|^2_{L^2}+\frac{1}{2\varepsilon_0\mu_0} \|B(0)\|^2_{L^2}. \nonumber
	\end{eqnarray}
\end{thm}
Next, we move to the problem of global existence. {}{Before going any further, we first need to rewrite the system as following and define some constants.
\begin{equation}\label{physic_model_1}
\begin{cases}
\partial_t v_--\mu_-\Delta\-+v_-\cdot \nabla v_-= - a_-(E+v_-\times B) + b_-(\+-\-)- \frac{1}{nm_-}\nabla p_- \\
\partial_t v_+-\mu_+\Delta\++v_+\cdot \nabla v_+=  a_+(E+v_+\times B) - b_+(\+-\-)- \frac{1}{nm_+}\nabla p_+ \\
\partial_t E = \frac{1}{\varepsilon_0\mu_0}\nabla\times B-\frac{ne}{\varepsilon_0}(Zv_+-v_-) \\
\partial_t B = -\nabla\times E \\
\text{div} v_-=\text{div} v_+=\text{div} B=\text{div} E=0,\\
\end{cases}
\end{equation}
where we set $ \mu_\pm=\frac{\nu_\pm}{nm_\pm} $, $ a_+=\frac{eZ}{m_+} $, $ a_-=\frac{e}{m_-} $, $ b_\pm=\frac{\alpha}{nm_\pm} $. Moreover, we define $ \mu=\min(\mu_-,\mu_+) $.\\
	    }
The following theorem shows the local-in-time well-posedness for \eqref{physic_model_1} in 3D case.
\begin{thm}\label{3dMain2}
		If the initial conditions of the physical model \eqref{physic_model} are such that
		{}{$(v_{-,0},v_{+,0},E_0,B_0)\in H^{1/2}\times H^{1/2}\times L^2\times L^2$}, then there exists $T>0$ and a unique solution $(v_-,v_+,E,B)$ of system \eqref{physic_model} such that
		{}{\begin{eqnarray*}
			v_\pm &\in& C([0,T]; H^{\frac{1}{2}}) \cap L^2(0,T; \dot{H}^{\frac{3}{2}}) \\
			E, \;B &\in& C([0,T]; L^2). \\
		\end{eqnarray*}}Furthermore, {}{if $ \|v_\pm(0)\|_{\dH^{1/2}}<\frac{\mu}{2{}{c_1}}$,}
		then the unique solution is global, satisfies the energy estimate and for any $ T>0 $ 
		{}{
			\[
			\|v_\pm\|_{C(0,T; {\dH}^{\frac{1}{2}})}<\frac{\mu}{2{}{c_1}},
			\]
			{}{where $ c_1 $ is a constant only depending on dimension.}
			}			
		\end{thm}
	\begin{rem}
		Since the Rayleigh friction term $R  $ does not improve the regularity of our system (See Lemma \ref{NS}), the results we get here are the same as for the classical Navier-Stokes equations.
	\end{rem}
\begin{rem}\label{cancel}
	If one assumes that all the above physical constants are normalized to one, then this would yield nice cancelations in the system and thus would avoid some extra technical difficulties. Indeed, in such a case, the system reads as follows		
	\begin{equation}\label{NSM1}
	\begin{cases}
	\partial_t v_-+v_-\cdot \nabla v_--\Delta v_-+\nabla p_-=-(E+v_-\times B)-R\\
	\partial_t v_++v_+\cdot \nabla v_+-\Delta v_++\nabla p_+=(E+v_+\times B)+R\\
	\partial_t E-\nabla\times B=-(v_+-v_-)\\
	\partial_t B+\nabla\times E=0\\
	R:=-\alpha(v_+-v_-)\\
	\text{div} v_-=\text{div} v_+=\text{div} B=\text{div} E=0\\
	\end{cases}
	\end{equation}
	with the initial data
	\begin{equation*}
	v_-|_{t=0}=v_{-,0},\quad v_+|_{t=0}=v_{+,0},\quad B|_{t=0}=B_0,\ E|_{t=0}=E_0.
	\end{equation*}
Defining the bulk velocity $u=\frac{v_-+v_+}{2}$ and the electrical
	current $j=\frac{v_+-v_-}{2}$ we can rewrite the system \eqref{NSM1}, as in \cite{AIM}, in the following equivalent form
	
	\begin{equation}\label{NSM2}
	\left\{
	\begin{array}{rcl}
	\partial_t u + u\cdot \nabla u + j\cdot \nabla j - \Delta u &=& - \nabla p + j\times B \\
	\partial_t j + u\cdot \nabla j + j\cdot \nabla u - \Delta j + 2\alpha j &=& - \nabla \bar{p} + E + u\times B \\
	\partial_t E - \nabla \times B &=& - 2j \\
	\partial_t B + \nabla \times E &=& 0 \\
	\text{div} u = \text{div} j = \text{div} E = \text{div} B &=& 0.
	\end{array}
	\right.
	\end{equation}
	System \eqref{NSM2} has appeared in \cite{AIM} including its dependence upon the speed of light that shows
	up in Maxwell's equations. The non-relativistic asymptotic of that system (i.e. as the speed of light tends to infinity), was recently
	analyzed and a convergence towards the standard Magneto Hydrodynamic equations was shown. We refer to \cite{AIM}
	for full details.
\end{rem}
{}{
\begin{rem}
	Another reason to keep all the physical parameters is the study of singular limits in forthcoming work.
\end{rem}
}
{Let $\mathbb{P}$ denotes} Leray projection onto divergence free vector fields. More precisely, if $u$ is a smooth vector field on $\mathbb{R}^d$, then $u$ can be uniquely written as the sum of divergence free vector $v$ and a gradient:
\[
u=v+\nabla \omega,\quad\text{div} v=0.
\]
Then, Leray projection of $u$ is $\mathbb{P}u=v$. Denote by $\Delta_q$ the frequency localization
operator, defined as follows:\\
Assume $\mathcal{C}$ is the ring of center $0$ with small radius 1/2 and great radius 2, $B$ is the ball centered at 0 with radius 1. There exist two nonnegative radial function $\chi \in \mathcal{D}(B)$ and $\phi \in \mathcal{D}(\mathcal{C})$, such that
\[
\chi(\xi)+\sum_{p\ge0}\phi(2^{-p}{}{\xi})=1.
\]
Then define
\[
\Delta_q u = \mathcal{F}^{-1}(\phi(2^{-q}{}{\xi})\mathcal{F}u),
\]
where $\mathcal{F}$ is the Fourier transform.\\
We define the space: $L^p(0,T;H^1)$, also denoted as $L^p_TH^1$ by its norm
\[
\|f(t,x)\|^p_{L^p_TH^1}:=\int^T_0\|f\|^p_{H^1}(t)dt < \infty.
\]

Finally, the spaces $\tilde{L}^r_TH^s$ and $\tilde{L}^r_T\dot{H}^s$ are the set of tempered distributions $u$ such that
\[
\|u\|_{\tilde{L}^r_TH^s} := \|(1+2^q)^s\|\Delta_q u\|_{L^r_TL^2}\|_{l^2(\mathbb{Z})} < \infty
\]
\[
\|u\|_{\tilde{L}^r_T\dot{H}^s} := \|2^{qs}\|\Delta_q u\|_{L^r_TL^2}\|_{l^2(\mathbb{Z})} < \infty.
\]
This kind of spaces were first introduced by Chemin and Lerner in \cite{CL}.
This paper is organized as follows. Section 2 is devoted to list some preliminary lemmas. In sections 3 to 5,
we give  the proofs of Theorem \ref{2dMain} to Theorem \ref{3dMain2}.

\section{Parabolic regularization, product estimates and energy estimates}
Here, we collect the main tools that will help us in the proofs of our results. {The first one concerns a parabolic regularization with or without friction term.}
\begin{lem}[parabolic regularization]\label{NS}
	Let $u$ be a smooth divergence-free vector field which solves
	\begin{equation}\label{PR1}
	\begin{array}{rcl}
	\partial_t u - \mu\Delta u + b\nabla p &=& f_1 + f_2,\quad \text{div} u=0 \\
	u|_{t=0} &=& u_0
	\end{array}
	\end{equation}
	or
	\begin{equation}\label{PR2}
	\begin{array}{rcl}
	\partial_t u + a u - \mu\Delta u + b\nabla p &=& f_1 + f_2,\quad \text{div} u=0 \\
	u|_{t=0} &=& u_0
	\end{array}
	\end{equation}
	on some interval $[0,T]$, where $ a $ and $ b $ are  nonnegative constants. Then, for every $p\ge r_1\ge 1$, $p\ge r_2\ge 1$ and $s\in \mathbb{R}$,
	\begin{eqnarray}\label{homo}
			\|u\|_{C([0,T];\dot{H}^s)\cap \tilde{L}^p_T \dot{H}^{s+2/p}}&\lesssim& (1+{\mu}^{-\frac{1}{p}})\|u_0\|_{\dot{H}^s} + (1+{\mu}^{-\frac{1}{p}}){\mu}^{-1+\frac{1}{r_1}}\|f_1\|_{\tilde{L}^{r_1}_T\dot{H}^{s-2+2/{r_1}}} \nonumber\\
			&&+(1+{\mu}^{-\frac{1}{p}}) {{\mu}}^{-1+\frac{1}{r_2}}\|f_2\|_{\tilde{L}^{r_2}_T\dot{H}^{s-2+2/{r_2}}}.
	\end{eqnarray}
	
	
	We also have a similar result in nonhomogeneous spaces but with $T$-dependent constants. More specifically,
	\begin{eqnarray}\label{nonhomo}
	\|u\|_{\tilde{L}^p_T H^{s+2/p}} &\lesssim& C_T \left( {\mu}^{-\frac{1}{p}}\|u_0\|_{{H}^s} + {\mu}^{-1-\frac{1}{p}+\frac{1}{r_1}}\|f_1\|_{\tilde{L}^{r_1}_T{H}^{s-2+2/{r_1}}}\right. \nonumber\\
	&& \left.+ {{\mu}}^{-1-\frac{1}{p}+\frac{1}{r_2}}\|f_2\|_{\tilde{L}^{r_2}_T{H}^{s-2+2/{r_2}}}\right).
	\end{eqnarray}
	where $C_T=C\max\{1,T\}$ with $C$ a universal constant.
\end{lem}

\begin{proof}
We only give a sketch of the proof in the case \eqref{PR2} with $f_1=f$ and $f_2=0$.
For more details of the proof, we refer to \cite{M}. The equation \eqref{PR2} can be written as the following:
	\[
	\partial_t u - (\mu\Delta -{a} I) u + b\nabla \pi = f.
	\]
	By Duhamel's formula, we have
	\begin{equation}\label{PR_Sol}
	u(t)=e^{t(\mu\Delta-{a}I)}u_0+\int_0^t e^{(t-t')(\mu\Delta-{a}I)}\mathbb Pf(t')dt'.
	\end{equation}
	Applying $\Delta_q$, the frequency localization operator, to \eqref{PR_Sol}, taking $L^2$ norm in space, and using the standard estimate for $\Delta_q$(see for instance \cite{M} and \cite{YZ}), we get
	\begin{eqnarray*}
		\|\Delta_q u(t)\|_{L^2} &\le& \|e^{-{a}t}\Delta_q u_0\|_{L^2} e^{-c2^{2q}\mu t} + \int_0^t e^{{}{-}c2^{2q}\mu(t-t')} \|e^{-{a}(t-t')}\Delta_q \mathbb Pf(t')\|_{L^2}dt' \\
		&\le& \|\Delta_q u_0\|_{L^2} e^{-c2^{2q}\mu t} + \int_0^t e^{{}{-}c2^{2q}\mu(t-t')} \|\Delta_q \mathbb Pf(t')\|_{L^2}dt'.
	\end{eqnarray*}
Then we can follow the same method which is used to get the estimate for \eqref{PR1}(See \cite{M}). Taking $L^p$ norm in time and using Young's inequality (in time) we obtain
	\begin{eqnarray*}
		\|\Delta_q u\|_{L^p_T L^2} &\lesssim& \mu^{-\frac{1}{p}}2^{-{\frac{2q}{p}}}\|\Delta_qu_0\|_{L^2} + \|e^{-c2^{2q}\mu t}1_{t>0}\|_{L^\alpha}\|\Delta_qf\|_{L^r_{{}{T}}L^2}\\
		&\lesssim& \mu^{-\frac{1}{p}}2^{-\frac{2q}{p}}\|\Delta_qu_0\|_{L^2} + \mu^{-\frac{1}{\alpha}}2^{-\frac{2q}{\alpha}}\|\Delta_q f\|_{L^r_{{}{T}}{L^2}},
	\end{eqnarray*}
	where $\frac{1}{p}+1=\frac{1}{\alpha}+\frac{1}{r}$. At last, multiplying by $2^{q(s+\frac{2}{p})}$ and taking $l^2$ norm over $q\in\mathbb{Z}$, we get the desired result.
\end{proof}
	\begin{rem}
		We should note the universal constant in the estimate is independent of $ a $. So in the proof of local existence, the small existence time $ T $ is independent of $ \alpha $ which appears in $ R $. While applying the lemma to physical model, we put $ R $ to the right hand side so that the universal constant will still be independent of $ \alpha $. 	
	\end{rem}
	The next Lemma is a standard energy estimate for the Maxwell's system. 

{}{\begin{lem}
	\label{MW}
	Let $f$ be in $L^1((0,T); H^s)$ for $s\in \mathbb R$, and $ a>0 $ be a constant. Then, Maxwell's equations 
	\[
	\left\{
	\begin{array}{rcl}
	\partial_t E - {}{a}\nabla \times B &=& {f}\\
	\partial_t B + \nabla \times E &=& 0 \\
	E(t=0) &=& E_0 \\
	B(t=0) &=& B_0
	\end{array}
	\right.
	\]
	
	has a unique solution $(E,B)\in C([0,T];H^s)$, and it satisfies 
	\[
\|E\|_{C([0,T];H^s)} + \sqrt{a}\|B\|_{C([0,T];H^s)} \le \|E_0\|_{H^s} + \sqrt{a}\|B_0\|_{H^s} + {\|f\|_{L^1_T H^s}}.
	\]
\end{lem}
\begin{proof}
The proof is straightforward (see for example \cite{M}). Applying $ H^s $ inner product with $ E $ to the first equation and with $ aB $ to the second equation, adding them implies
\[
	\frac{1}{2}\frac{d}{dt}\|E\|^2_{H^s}+\frac{a}{2}\frac{d}{dt}\|B\|^2_{H^s} = (f,E)_{H^s},
\]
where $ (\cdot,\cdot)_{H^s} $ stands for $ H^s $ inner product. Let $ F=(E,\sqrt{a}B) $, $ g=(f,0) $, the above identity becomes 
\[
	\frac{1}{2}\frac{d}{dt}\|F\|^2_{H^s}=(g,F)_{H^s}.
\]
Applying Cauchy Schwarz inequality yields
\[
\frac{1}{2}\frac{d}{dt}\|F\|^2_{H^s}\le\|g\|_{H^s}\|F\|_{H^s},
\]
which is equivalent to
\[
\frac{d}{dt}\|F\|_{H^s}\le\|g\|_{H^s}.
\]
Then integrating in time proves the inequality. For the time continuity part, we rewrite the system as the following equivalent form
	\[
	\left\{
	\begin{array}{rcl}
	\partial_t E - {}{\sqrt{a}}\nabla \times \sqrt{a}B &=& {f}\\
	\partial_t \sqrt{a}B + \sqrt{a}\nabla \times E &=& 0 \\
	E(t=0) &=& E_0 \\
	B(t=0) &=& B_0.
	\end{array}
	\right.
	\]
	Then Duhamel's formula gives us
	\[
	F(t)=e^{tL}F(0)+\int_0^te^{(t-\tau)L}g(\tau)d\tau,
	\]
	with $ L=\left(\begin{array}{cc}
		0&-\sqrt{a}\nabla\times\\
		\sqrt{a}\nabla\times&0
	\end{array} \right)$. 
Since the semigroup generated by a skew-symmetric operator is a continuous semigroup, then the solution $ F(t) $ is continuous in time.

\end{proof}}
Next, we set the nonlinear estimates necessary to derive the {\it a priori} bounds.
\begin{lem}[products estimate]\label{product}
	For $ 0<s<{}{d/2} $, and $ u,v,B $ are functions of $ x $, we have
		\begin{equation}\label{est1}
		\|uv\|_{\dH^{s-d/2}}\lesssim \|u\|_{H^s}\|v\|_{L^2}.
		\end{equation}
		{}{
			\begin{equation}\label{estExtra}
				\|uv\|_{H^{2s-d/2}}\lesssim\|u\|_{H^s}\|v\|_{H^s}.
			\end{equation}}In particular, when $ d=2 $, it holds that
		\begin{equation}\label{est3}
		\|u\nabla v\|_{H^{s-1}} \lesssim \|u\|_{L^2}\|v\|_{H^1}+\|u\|_{H^1}\|v\|_{\dH^1}.
		\end{equation}
	While $  d=3 $, one has
	\begin{eqnarray}
	\label{PE1}
	\|u \cdot \nabla v\|_{L^2_T{H}^{-\frac{1}{2}}} & \lesssim& \|u\|_{L^\infty_T {H}^{\frac{1}{2}}}\|v\|_{L^2_T {H}^{\frac{3}{2}}} \\
	\label{PE2}
	\|u \cdot \nabla v\|_{\tilde{L}^{\frac{4}{3}}_T L^2} & \lesssim& \|u\|_{L^\infty_T {H}^{\frac{1}{2}}}^{\frac{1}{2}} \|u\|_{L^2_T {H}^{\frac{3}{2}}}^{\frac{1}{2}} \|v\|_{L^2_T {H}^{\frac{3}{2}}} \\
	\label{PE3}
	\label{PE4}
	\|u \times B\|_{L^2_T{H}^{-\frac{1}{2}}} & \lesssim& T^{\frac{1}{4}} \|u\|_{L^\infty_T {H}^{\frac{1}{2}}}^{\frac{1}{2}} \|u\|_{L^2_T {H}^{\frac{3}{2}}}^{\frac{1}{2}} \|B\|_{L^\infty_T L^2}.
	\end{eqnarray}
\end{lem}
\begin{proof}
For \eqref{est1}, H\"older's inequality and Sobolev embedding give that
	\begin{eqnarray*}
		\|uv\|_{\dot{H}^{s-d/2}} &\le& \|uv\|_{L^{\frac{d}{d-s}}} \\
		&\lesssim& \|u\|_{L^{\frac{2d}{d-2s}}}\|v\|_{L^2} \\
		&\lesssim& \|u\|_{H^s}\|v\|_{L^2}.
	\end{eqnarray*}
	{}{
Estimate \eqref{estExtra} is classic and we refer for example to \cite{BCD}, Corollary 2.55.} For \eqref{est3}, estimating the low and high frequencies separately yields
	\begin{eqnarray*}
		\|u\nabla v\|^2_{H^{s-1}} &=& \sum_{q}{}{(1+2^{q})^{2(s-1)}} \|\dq(u\nabla v)\|^2_{L^2}\\
		&\lesssim& \sum_{q\le 0}  \|\dq(u\nabla v)\|^2_{L^2} + \sum_{q\ge1} 2^{2q(s-1)}\|\dq(u\nabla v)\|^2_{L^2} \\
		&\lesssim&\sum_{q\le 0}  \|\dq(u\nabla v)\|^2_{L^2} +\|u\nabla v\|_{\dH^{s-1}}^2 \\
		&\lesssim&\sum_{q\le 0}  \|\dq(u\nabla v)\|^2_{L^2}+\|u\|_{H^s}^2\|v\|_{\dH^1}^2,
	\end{eqnarray*}where \eqref{est1} is applied in the last step.\\ 
Hence we only need to estimate $ \sum_{q\le 0}  \|\dq(u\nabla v)\|^2_{L^2} $. Thanks to Bony decomposition (see for example \cite{BCD}), we have
		\[\begin{aligned}
		\|\dq(u\nabla v)\|^2_{L^2} \lesssim &\sum_{|k-q|\le 2}\|S_{{}{k}-1} u \Delta_k(\nabla v)\|^2_{L^2} + \sum_{|k-q|\le 2}\|S_{{}{k}-1}(\nabla v) \Delta_k u\|_{L^2}^2 \\
		&+ \|{}{\sum_{\begin{subarray}{c}
				k\ge q+3\\|k-l|\le1
				\end{subarray}}}
		\dq(\Delta_k u\Delta_{{}{l}} (\nabla v))\|_{L^2}^2,\end{aligned}
		\]where $S_q=\sum_{k\leq q-1}\Delta_k$.
	Applying Bernstein's lemma (noting that $ d=2 $) and using $ q\le 0 $ gives {}{
	\[
	\sum_{|k-q|\le 2}\|S_{{}{k}-1}u \dk (\nabla v)\|_{L^2}^2 \lesssim \sum_{|k-q|\le 2}\|S_{{}{k}-1} u\|_{L^2}^2 2^{4k} \|\dk v\|_{L^2}^2 \lesssim \|u\|_{L^2}^2{}{\sum_{|k-q|\le 2}}\|{}{\Delta_k} v\|_{L^2}^2,
	\]}and therefore,{}{
	\begin{equation}\label{temp1}
	\sum_{q\le0} \sum_{|k-q|\le 2}\|S_{{}{k}-1}u \dk (\nabla v)\|_{L^2}^2\lesssim \|u\|_{L^2}^2\|v\|_{L^2}^2.
	\end{equation}}Again applying Bernstein's lemma and the fact that $ q\le 0 $ gives 
	\[\begin{aligned}
	\sum_{|k-q|\le 2}\|S_{{}{k}-1}\nabla v \dk u\|_{L^2}^2 &\lesssim \sum_{|k-q|\le 2}\|S_{{}{k}-1} (\nabla v)\|_{L^2}^2 2^{2k}\|\dk u\|_{L^2}^2 \\ &\lesssim \|v\|_{\dH^1}^2{}{\sum_{|k-q|\le 2}}\|{}{\Delta_k} u\|_{L^2}^2,\end{aligned}
	\]yielding{}{
	\begin{equation}\label{temp2}
	\sum_{q\le0} \sum_{|k-q|\le 2}\|S_{{}{k}-1}\nabla v \dq u\|_{L^2}^2 \lesssim \|u\|^2_{L^2}\|v\|^2_{\dH^1},
	\end{equation}as desired.} For the last term, applying Bernstein's lemma and the H\"older inequality implies
\[
	\begin{aligned}
	&\|{}{\sum_{\begin{subarray}{c}
			k\ge q+3\\|k-l|\le1
			\end{subarray}}}\dq(\Delta_k u\Delta_{{}{l}} (\nabla v))\|_{L^2}\\
		\le\,&{}{\sum_{\begin{subarray}{c}
				k\ge q+3\\|k-l|\le1
				\end{subarray}}} 2^{q}\|\Delta_k u \Delta_{{}{l}} (\nabla v)\|_{L^1}\\
		\lesssim\,&{}{\sum_{\begin{subarray}{c}
				k\ge q+3\\|k-l|\le1
				\end{subarray}}} 2^{q+{}{l}}\|\Delta_k u\|_{L^2} \|\Delta_{{}{l}} v\|_{L^2}\\
		\lesssim\,&{}{\sum_{\begin{subarray}{c}
				k\ge q+3\\|k-l|\le1
				\end{subarray}}} 2^{q+{}{k}}\|\Delta_k u\|_{L^2} \|\Delta_{{}{l}} v\|_{L^2}.
	\end{aligned}
\]
	Therefore, Young's and H\"older's inequalities give
	\begin{eqnarray}\label{temp3}
	\sum_{q\le0}\|{}{\sum_{\begin{subarray}{c}
			k\ge q+3\\|k-l|\le1
			\end{subarray}}}\dq(\Delta_k u\Delta_{{}{l}} (\nabla v))\|_{L^2}^2 &=&\left(\|{}{\sum_{\begin{subarray}{c}
			k\ge q+3\\|k-l|\le1
			\end{subarray}}}\dq(\Delta_k u\Delta_{{}{l}} (\nabla v))\|_{L^2}\right)_{l^2(q\le 0)}^2 \nonumber \\
	&\lesssim& \left({}{\sum_{\begin{subarray}{c}
			k\ge q+3\\|k-l|\le1
			\end{subarray}}}2^{2k}\|\Delta_k u\|_{L^2} \|\Delta_{{}{l}} v\|_{L^2}2^{q-k}\right)_{l^2(q\le 0)}^2 \nonumber \\
	&\lesssim& \left(2^{2q}\|\dq u\|_{L^2} {}{\sum_{|q-l|\le1}}\|\Delta_{{}{l}} v\|_{L^2}\right)_{l^1}^2(2^q)_{l^2(q\le 0)}^2 \nonumber \\
	&\lesssim& (2^q\|\dq u\|_{L^2} 2^q{}{\sum_{|q-l|\le1}}\|\Delta_{{}{l}} v\|_{L^2})_{l^1}^2 \nonumber\\
	&\lesssim& \|u\|_{\dH^1}^2 \|v\|_{\dH^1}^2.
	\end{eqnarray}
	\eqref{temp1}, \eqref{temp2} and \eqref{temp3} together yields
	\[
	\sum_{q\le 0}  \|\dq(u\nabla v)\|^2_{L^2} \lesssim \|u\|_{L^2}^2\|v\|_{H^1}^2+\|u\|_{\dH^1}^2 \|v\|_{\dH^1}^2.
	\]
	Hence
	\begin{eqnarray*}
		\|u\nabla v\|_{H^{s-1}} &\lesssim& \|u\|_{L^2}\|v\|_{H^1}+\|u\|_{H^s}\|v\|_{\dH^1}+\|u\|_{\dH^1} \|v\|_{\dH^1} \\
		&\lesssim& \|u\|_{L^2}\|v\|_{H^1}+\|u\|_{H^1}\|v\|_{\dH^1}.
	\end{eqnarray*}
	For the last {}{four estimates}, we only show the proof of \eqref{PE2} as the others are similar. Noting that $ d=3 $, we have
	\begin{eqnarray*}
		\|u \cdot \nabla v\|_{\tilde{L}^{\frac{4}{3}}_T L^2} & \lesssim& \|u \cdot \nabla v\|_{L^{\frac{4}{3}}_T L^2} \\
		& \lesssim& \|u\|_{L^4_T L^6} \|\nabla v\|_{L^2_T L^3} \\
		& \lesssim& \|u\|_{L^4_T {H}^1} \|\nabla v\|_{L^2_T {H}^{\frac{1}{2}}}.
	\end{eqnarray*}
	
	Interpolation gives us
	\begin{eqnarray*}
		\|u \cdot \nabla v\|_{\tilde{L}^{\frac{4}{3}}_T L^2} & \lesssim& \left(\int^T_0 \|u(t)\|_{{H}^1}^4dt\right)^{\frac{1}{4}} \|\nabla v\|_{L^2_T {H}^{\frac{1}{2}}} \\
		& \lesssim& \left(\int^T_0 \|u(t)\|_{{H}^{\frac{1}{2}}}^2 \|u(t)\|_{{H}^{{}{\frac{3}{2}}}}^2dt\right)^{\frac{1}{4}} \|\nabla v\|_{L^2_T {H}^{\frac{1}{2}}}, \\
	\end{eqnarray*}
	and using H\"older inequality, we obtain
	\begin{eqnarray*}
		\|u \cdot \nabla v\|_{\tilde{L}^{\frac{4}{3}}_T L^2}& \lesssim& \left(\|u\|^2_{L^\infty_T {H}^{\frac{1}{2}}} \|u\|^2_{L^2_T {H}^{\frac{3}{2}}}\right)^{\frac{1}{4}} 
		{\|v\|_{L^2_T {H}^{\frac{3}{2}}}} \\
		& \lesssim& \|u\|_{L^\infty_T {H}^{\frac{1}{2}}}^{\frac{1}{2}} \|u\|_{L^2_T {H}^{\frac{3}{2}}}^{\frac{1}{2}} 
		{\|v\|_{L^2_T {H}^{\frac{3}{2}}}}.
	\end{eqnarray*}
\end{proof}
The next lemma is a standard energy identity for the whole NSM system.
\begin{lem}\label{energyEst}
	For system \eqref{physic_model}, we have the following energy identity.
	\begin{eqnarray}\label{energy}
	\frac{nm_-}{2\varepsilon_0}\|v_-\|^2_{L^2} + \frac{nm_+}{2\varepsilon_0}\|v_+\|^2_{L^2} + \frac{1}{2} \|E\|^2_{L^2}+\frac{1}{2\varepsilon_0\mu_0} \|B\|^2_{L^2} && \\
	+\frac{\nu_-}{\varepsilon_0}\|v_-\|^2_{L^2_t\dot{H}^1} + \frac{\nu_+}{\varepsilon_0}\|v_+\|^2_{L^2_t\dot{H}^1} + \frac{\alpha}{\varepsilon_0}\|v_--v_+\|^2_{L^2_tL^2} &&\nonumber \\
	=&& \nonumber\\
	\frac{nm_-}{2\varepsilon_0}\|v_-(0)\|^2_{L^2} + \frac{nm_+}{2\varepsilon_0}\|v_+(0)\|^2_{L^2} + \frac{1}{2} \|E(0)\|^2_{L^2}+\frac{1}{2\varepsilon_0\mu_0} \|B(0)\|^2_{L^2}. \nonumber
	\end{eqnarray}
\end{lem}
\begin{proof}
	The proof is standard. Multiply $\frac{v_-}{\varepsilon_0}$, $\frac{v_+}{\varepsilon_0}$, $E$, $\frac{B}{\varepsilon_0\mu_0}$ to the first four equations of \eqref{physic_model} respectively, integrate over space and add them together. With the divergence free condition, we can get the desired result.
\end{proof}
\section{Global Well-Posedness in 2D Case}
In this section we prove Theorem \ref{2dMain}. The main idea is to apply the classical compactness method to prove the existence of a global weak solution then prove the uniqueness of the weak solution so that we obtain the unique strong solution (c.f. \cite{M}). The whole proof goes into three steps. Firstly, we provide an {\it a priori} estimate.\\
{\sc Step 1:  A priori estimate.}\\
The {\it a priori} estimate is given by the following lemma.
\begin{lem}\label{pri2d}
	If $ (\-,\+,E,B) $ solves \eqref{physic_model} on $ [0,T] $, then for any $ 0<s_1<1 $ the following estimate holds:
	\[
	\|v\|_{L^1_T H^{{}{s_1+1}}} \lesssim C_0C_T^2.
	\]
	Here $ C_0=\|\-\|_{L^2}+\|\+\|_{L^2}+\|E_0\|_{L^2}+\|B_0\|_{L^2} $, and $ v=v_\pm $.
\end{lem}
\begin{proof}For simplicity, let $v=v_\pm$, $C_T=C\max\{1,T\}$ with $C$ a universal constant and $ v'=v_\mp $. By the energy identity, 
	\[
	\|v\|_{L^\infty_TL^2}+\|v\|_{L^2_T\dH^1} \lesssim \|v_0\|_{L^2} +\|v'_0\|_{L^2}+\|E_0\|_{L^2}+\|B_0\|_{L^2}\lesssim C_0.
	\]
	Now fix $ s_1 $, there exists $ s $ such that $ 0<s_1<s<1 $. Thanks to estimate \eqref{nonhomo} in Lemma \ref{NS} (replacing $ s $ by $ s-1 $, let $ p_1=p_2=r_1=r_2=1 $), 
\begin{multline*}
	\|v\|_{\tL^1_T H^{s+1}} \le 
C_T(\|v_0\|_{H^{s-1}}+ \|v'\|_{\tL^1_TH^{s-1}}\\
+\|E\|_{\tL^1_TH^{s-1}}+\|v\times B\|_{\tL^1_T H^{s-1}}+\|v\nabla v\|_{\tL^1_T H^{s-1}}).
\end{multline*}
	The fact that $ 0<s_1<s $, $ \|v\|_{L^1_T H^{s_1+1}}\le\|v\|_{\tilde{L}^1_T H^{s+1}} $ together with energy identity implies
	\begin{eqnarray*}
		\|v\|_{L^1_T H^{s_1+1}} &\le& C_T(\|v_0\|_{L^2}+ T\|v'\|_{L^\infty_TL^2}\\
& &
\ \ \ \ \ 
+T\|E\|_{L^\infty_TL^2}+\|v\times B\|_{L^1_T H^{s-1}}+\|v\nabla v\|_{L^1_T H^{s-1}}) \\
		&\le&C_T(C_0+ 2TC_0+\|v\times B\|_{L^1_T H^{s-1}}+\|v\nabla v\|_{L^1_T H^{s-1}}).
	\end{eqnarray*}
	Estimates \eqref{est1} and \eqref{est3} in Lemma \ref{product} give the following estimates on these nonlinear terms,
\begin{multline*}
	\|v\times B\|_{L^1_T H^{s-1}} \lesssim \|v\|_{L^1_TH^1} \|B\|_{L^\infty_T L^2}\\
 \lesssim (T\|v\|_{L^\infty_TL^2}+T^{1/2}\|v\|_{L^2\dH^1})\|B\|_{L^\infty_TL^2}\lesssim(T+T^{1/2})C_0^2,
\end{multline*}
	\begin{eqnarray*}
		\|v\nabla v\|_{L^1_T H^{s-1}} &\lesssim& \|v\|_{L^\infty_TL^2} \|v\|_{L^1_TH^1} +\|v\|_{L^2H^1}\|v\|_{L^2\dH^1} \\
		&\lesssim& (T\|v\|_{L^\infty_TL^2}+T^{1/2}\|v\|_{L^2\dH^1}) \|v\|_{L^\infty_TL^2} +\|v\|_{L^2H^1}\|v\|_{L^2\dH^1} \\
		&\lesssim& (T\|v\|_{L^\infty_TL^2}+T^{1/2}\|v\|_{L^2\dH^1}) \|v\|_{L^\infty_TL^2}\\
& &\ \ \ \ \ 
 +(T^{1/2}\|v\|_{L^\infty_TL^2}+\|v\|_{L^2\dH^1})\|v\|_{L^2\dH^1} \\
		&\lesssim& (1+2T^{1/2}+T)C_0^2.
	\end{eqnarray*}
	Hence 
	\[
	\|v\|_{L^1_T H^{s_1+1}} \lesssim C_T(C_0+2C_0T+C_0^2(1+3T^{1/2}+2T))\lesssim C_0C_T^2.
	\]
\end{proof}
{\sc Step 2:  A compactness argument.}\\
We will approximate the original system by a frequency cutoff system and apply the classical compactness argument to pass the limit (c.f \cite{M}). Let us define a frequency cutoff operator $ J_k $ by
\[
J_k u := \mathcal{F}^{-1}(1_{B(0,k)}(\xi)\hat{u}(\xi)),
\]
where $ \mathcal{F} $ and $ \hat{\cdot} $ are the Fourier transform in the space variable. We consider the following approximating system
\begin{equation}\label{approxi}
	\left\{\begin{aligned}
	&nm_-\partial_t v_-^k=\nu_-\Delta J_kv_-^k - nm_- J_k(J_kv_-^k\cdot \nabla J_kv_-^k) - en(E^k+J_k(J_kv_-^k\times J_kB^k)) \\
	&\qquad\qquad\qquad- R^k- \nabla p_-^k \\
	&nm_+\partial_t v_+^k=\nu_+\Delta J_kv_+^k - nm_+ J_k(J_kv_+^k\cdot \nabla J_kv_+^k) + eZn(E^k+J_k(J_kv_+^k\times J_kB^k)) \\
	&\qquad\qquad\qquad+ R^k- \nabla p_+^k \\
	&\partial_t E^k =\frac{1}{\varepsilon_0\mu_0}\nabla\times J_kB^k-\frac{ne}{\varepsilon_0}(Zv_+^k-v_-^k) \\
	&\partial_t B^k = -\nabla\times J_kE^k \\
	&R^k:=-\alpha(v_+^k-v_-^k)\\
	&\text{div} v_-^k=\text{div} v_+^k=\text{div} B^k=\text{div} E^k=0\\
	\end{aligned}\right.
	\end{equation}
with initial data,
\[
v_-^k(t=0) = J_k(v_-(t=0)),v_+^k(t=0) = J_k(v_+(t=0)),
\]
\[
E^k(t=0) = J_k(E(t=0)),B^k(t=0) = J_k(B(t=0)).
\]
The above system is now an ODE that has a unique solution $ (v_-^k,v_+^k,E^k,B^k)\in C([0,T_k];L^2) $ with a positive maximal time of existence $ T_k $. Since $ J_k^2=J_k $, $ J_k(v_-^k,v_+^k,E^k,B^k) $ is also a solution. Hence uniqueness implies $ J_k(v_-^k,v_+^k,E^k,B^k)=(v_-^k,v_+^k,E^k,B^k) $ and therefore we can get rid of $ J_k $ in front of $ (v_-^k,v_+^k,E^k,B^k) $ and only leave $ J_k $ in front of nonlinear terms:
\begin{equation}\label{approxi_ver2}
\begin{cases}
nm_-\partial_t v_-^k=\nu_-\Delta v_-^k - nm_- J_k(v_-^k\cdot \nabla v_-^k) - en(E^k+J_k(v_-^k\times B^k)) - R^k- \nabla p_-^k \\
nm_+\partial_t v_+^k=\nu_+\Delta v_+^k - nm_+ J_k(v_+^k\cdot \nabla v_+^k) + eZn(E^k+J_k(v_+^k\times B^k)) + R^k- \nabla p_+^k \\
\partial_t E^k = \frac{1}{\varepsilon_0\mu_0}\nabla\times B^k-\frac{ne}{\varepsilon_0}(Zv_+^k-v_-^k) \\
\partial_t B^k = -\nabla\times E^k \\
R^k:=-\alpha(v_+^k-v_-^k)\\
\text{div } v_-^k=\text{div } v_+^k=\text{div} B^k=\text{div} E^k=0.\\
\end{cases}
\end{equation}
Now we prove that actually $ T_k=\infty $. We will see that the energy identity still holds for \eqref{approxi_ver2},
\begin{eqnarray}\label{energy_approxi}
\frac{nm_-}{2\varepsilon_0}\|v_-^k\|^2_{L^2} + \frac{nm_+}{2\varepsilon_0}\|v_+^k\|^2_{L^2} + \frac{1}{2} \|E^k\|^2_{L^2}+\frac{1}{2\varepsilon_0\mu_0} \|B^k\|^2_{L^2} && \\
+\frac{\nu_-}{\varepsilon_0}\|v_-^k\|^2_{L^2_t\dot{H}^1} + \frac{\nu_+}{\varepsilon_0}\|v_+^k\|^2_{L^2_t\dot{H}^1} + \frac{\alpha}{\varepsilon_0}\|v_-^k-v_+^k\|^2_{L^2_tL^2} &&\nonumber \\
=&& \nonumber\\
\frac{nm_-}{2\varepsilon_0}\|J_kv_-(0)\|^2_{L^2} + \frac{nm_+}{2\varepsilon_0}\|J_kv_+(0)\|^2_{L^2} + \frac{1}{2} \|J_kE(0)\|^2_{L^2}+\frac{1}{2\varepsilon_0\mu_0} \|J_kB(0)\|^2_{L^2}. \nonumber
\end{eqnarray}
Hence the $ L^2 $ norm of $ (v_-^k,v_+^k,E^k,B^k) $ is bounded uniformly in time and we get $ T_k=\infty $. 
Moreover,  a priori estimate we get in the previous section also holds, i.e, for any $ T>0 $
\[
\|v^k_\pm\|_{C([0,T];L^2)\cap L^2_T\dH^1 }+\|E^k\|_{C([0,T];L^2)} + \|B^k\|_{C([0,T];L^2)} \lesssim C_0,
\]
and 
\[
\|v^k\|_{L^1_T H^{s_1+1}} \lesssim C_0C_T^2.
\]
where $ C_0=\|\-\|_{L^2}+\|\+\|_{L^2}+\|E_0\|_{L^2}+\|B_0\|_{L^2} $, $ v=v_\pm $.\\
To apply the compactness argument, we need to bound $ \partial_t(v_-^k,v_+^k) $ in $ L^{{}{2}}_TH^{-3/2} $ uniformly in $ k $. 
Applying Lemma \ref{product} we obtain
\begin{eqnarray*}
		\|\partial_tv^k\|_{L^{{}{2}}_TH^{{}{-3/2}}} &\lesssim& \|\Delta v^k\|_{L^{{}{2}}_TH^{{}{-3/2}}} + \|E^k+v^k+v'^k\|_{L^{{}{2}}_TH^{{}{-3/2}}} \\
		&&+\|v^k\cdot \nabla v^k\|_{L^{{}{2}}_TH^{{}{-3/2}}}+\|v^k\times B^k\|_{L^{{}{2}}_TH^{{}{-3/2}}} \\
		&\lesssim& \|v^k\|_{L^2_{{}{T}}\dH^1} +T^{{}{1/2}}C_0+\|{}{v^k\otimes v^k}\|_{L^{{}{2}}_TH^{-1/2}} +  \|v^k\times B^k\|_{L^{{}{2}}_TH^{s-1}}  \\
		&\lesssim& C_0 +T^{{}{1/2}}C_0+{}{\|v^k\|_{L^\infty_TL^2}\|v^k\|_{L^2_{{}{T}}H^{1/2}}}+  \|v^k\|_{L^2_TH^1}\|B^k\|_{L^{{}{\infty}}_TL^2}.
\end{eqnarray*}
By energy estimate, we have
\[
{}{\|v^k\|_{L^2_{T}H^{1/2}}}\le\|v^k\|_{L^2_TH^1} \le (T^{1/2}\|v^k\|_{L^\infty_TL^2}+\|v^k\|_{L^2_{{}{T}}`\dH^1}) \lesssim (T^{1/2}+1)C_0,
\]
leading to the bound on $ \|\partial_tv^k\|_{L^{{}{2}}_TH^{{}{-3/2}}} $:
\[
\|\partial_tv^k\|_{L^{{}{2}}_TH^{{}{-3/2}}}\lesssim (1+T^{1/2})(2C_0^2+C_0).
\]
{}{Next we introduce Aubin-Lions-Simon Lemma (see for example \cite{BF}, \cite{SIMON}, \cite{L} and \cite{Teman}). The proof is given in \cite{BF}.
	\begin{lem}[Aubin-Lions-Simon]
		Let $ X\subset Y\subset Z $ be three Banach spaces. We assume that the embedding of $ Y $ in $ Z $ is continuous and that the embedding of $ X $ in $ Y $ is compact. Let $ p,r $ be such that $ 1\le p,r\le \infty $. For $ T>0$, we define
		\[
			E_{p,r}=\left\{u\in L^p_TX,\partial_tv\in L^r_TZ \right\}.
		\]
		i) If $ p< \infty$, the embedding of $ E_{p,r} $ in $ L^p_TY $ is compact.\\
		ii) If $ p=\infty $ and $ r>1 $, the embedding of $ E_{p,r} $ in $ C([0,T];Y) $ is compact.
	\end{lem}	
Before applying this lemma, we summarize the work above. Actually, we have shown
\[
	\|v^k\|_{L^1_TH^{s_1+1}\cap C([0,T];L^2) \cap L^2_T\dH^1}\lesssim\max(C_0C^2_T,C_0),
\]
\[
\|\partial_tv^k\|_{L^{{}{2}}_TH^{{}{-3/2}}}\lesssim (1+T^{1/2})(2C_0^2+C_0),
\]
and
\[
	\|E^k\|_{C([0,T];L^2)} + \|B^k\|_{C([0,T];L^2)} \lesssim C_0.
\]
Let $ B_n:=\{x: |x|\le n\} $. Then applying Aubin-Lions-Simon Lemma gives us that  $ \{v^k\} $ is compact in $L^1(0,T;H^{s_1'+1 }(B_n))\cap L^2(0,T;L^2(B_n)) $ with $ 0\le s_1'< s_1 $. 
Hence, there exist $$ v_\pm\in L^1(0,T;H^{s_1'+1 }(B_n))\cap L^2(0,T;L^2(B_n)),\quad E, B\in L^\infty(0,T;L^2(B_n))$$ and a subsequence $ k_m $ such that as $ m\to\infty$,
\[
	v_\pm^{k_m}\to v_\pm \text{ strongly in $ L^1(0,T;H^{s_1'+1 }(B_n))\cap L^2(0,T;L^2(B_n)) $},
\]
\[
	(E^{k_m},B^{k_m})\rightharpoonup (E, B) \text{ weakly-* in $ L^\infty(0,T;L^2(B_n)) $}.
\]
Indeed, by a diagonal extraction argument the above strong convergences are true for all $ B_n, n>0 $. Therefore we can pass the limit in \eqref{approxi_ver2} and obtain a weak solution.\\
The continuity of $ v_\pm $ is obtained with the fact that $ v_\pm $ is unique and solves Navier-Stokes equations. Consider $ -en(E+v_-\times B)-R $ as a body force of the following Navier-Stokes equation
\[
	nm_-\partial_tv_--\nu_-\Delta v_-+nm_-(v_-\cdot\nabla)v_-+\nabla p_-=-en(E+v_\times B)-R.
\]
Since $ v_- $ is unique in energy space $L^\infty_TL^2 \cap L^2_T\dH^1  $ which will be proved in next step, $ v_- $ also satisfies that $ v_-\in C([0,T];L^2) $. The proof of continuity of $ v_+ $ is the same. The continuity of $ (E, B) $ is obtained after applying Lemma \ref{MW}.
}\\
{\sc Step 3:  Uniqueness of solutions.}\\
Here we prove the uniqueness of solutions to \eqref{physic_model} $ (\+,\-,E,B) $ in $$ L^\infty_TL^2 \cap L^2_T{}{H^1} \times L^\infty_TL^2 \cap L^2_T{}{H^1} \times L^\infty_T L^2 \times  L^\infty_T L^2.$$ 
Take $ (\+^1,\-^1,E^1,B^1) $ and $(\+^2,\-^2,E^2,B^2) $ are two solutions of \eqref{physic_model}. Letting $ (\+,\-,E,B)=(\+^1,\-^1,E^1,B^1)-(\+^2,\-^2,E^2,B^2) $, then $ (\+,\-,E,B)$ satisfies the following system with zero initial datum:
\begin{equation}\label{uniq_sysm}
\begin{cases}
nm_-\partial_t v_-=\nu_-\Delta v_- - nm_- v_-^2\cdot \nabla v_--nm_- v_-\cdot \nabla v_-^1 \\
\qquad\qquad\qquad- en(E+v_-^2\times B+v_-\times B^1) - R- \nabla p_- \\
nm_+\partial_t v_+=\nu_+\Delta v_+ - nm_+ v_+^2\cdot \nabla v_+ - nm_+ v_+\cdot \nabla v_+^1 \\
\qquad\qquad\qquad+ eZn(E+v_+^2\times B+v_+\times B^1) + R- \nabla p_+ \\
\partial_t E = \frac{1}{\varepsilon_0\mu_0}\nabla\times B-\frac{ne}{\varepsilon_0}(Zv_+-v_-) \\
\partial_t B = -\nabla\times E \\
R:=-\alpha(v_+-v_-)\\
\text{div} v_-=\text{div} v_+=\text{div} B=\text{div} E=0.\\
\end{cases}
\end{equation}
{}{
	Recall that $ v=v_\pm $, $ v'=v_\mp $. Let $ X=L^\infty_TL^2\cap L^2_TH^1$ so that $ v_\pm,\v\in X $ and let $q_1>0$, ${0<s'<1}$ be such that $ \frac{1}{q_1}=\frac{1+s'}{2} $. Next we apply Lemma \ref{NS} to estimate $ \|v\|_X $.
	For $ \|v\|_{L^\infty_TL^2} $ we choose $ p=\infty, s=0, r_1=r_2=q_1 $ in the lemma, for $ \|v\|_{L^2_TH^1} $ we choose $ p=2, s=0, r_1=r_2=q_1 $.
	Finally one obtains
	\begin{equation}\label{estTEMP}
		\|v\|_X\lesssim C_T\|E+v'+v^2\times B+v\times B^1+v^2\nabla v+ v\nabla v^1\|_{L^{q_1}H^{s'-1}}
	\end{equation}
 where $ C_T=C\max(1,T) $.
	Since $ s'-1<0 $, by \eqref{est1} one has
	\begin{eqnarray}\label{v2b}
		\|v^2\times B\|_{L^{q_1}_TH^{s'-1}}&\le&\|v^2\times B\|_{L^{q_1}_T\dH^{s'-1}}\nonumber\\
		&\lesssim&{}{\|v^2\|_{L^{q_1}_T H^{s'}}\|B\|_{L^{\infty}_TL^2}}\nonumber\\
		&\lesssim& T^{\frac{1}{q_1}-\frac{1}{2}}\|v^2\|_{L^2_TH^1}\|B\|_{L^\infty_TL^2}\nonumber\\
		&\lesssim&T^{\frac{1}{q_1}-\frac{1}{2}}\|v^2\|_{X}\|B\|_{L^\infty_TL^2}.
	\end{eqnarray}
	Similarly it holds that
	\begin{equation}\label{vb1}
		\|v\times B^1\|_{L^{q_1}_TH^{s'-1}}\lesssim T^{\frac{1}{q_1}-\frac{1}{2}}\|B^1\|_{L^\infty_TL^2}\|v\|_{X}.
	\end{equation}
	With the help of \eqref{estExtra} we have
	\begin{equation}\label{v2v}
		\|v^2\nabla v\|_{L^{q_1}_TH^{s'-1}}\lesssim\|v^2\|_{L^{2q_1}_TH^{\frac{s'+1}{2}}}\|v\|_{L^{2q_1}_TH^{\frac{s'+1}{2}}}\lesssim\|v^2\|_{{}{L^{2q_1}_TH^{\frac{s'+1}{2}}}}\|v\|_X,
	\end{equation}
	and
	\begin{equation}\label{vv1}
		\|v\nabla v^1\|_{L^{q_1}H^{s'-1}}\lesssim\|v^1\|_{L^{2q_1}_TH^{\frac{s'+1}{2}}}\|v\|_{L^{2q_1}_TH^{\frac{s'+1}{2}}}\lesssim\|v^1\|_{{}{L^{2q_1}_TH^{\frac{s'+1}{2}}}}\|v\|_X,
	\end{equation}
	where in last step we use the embedding $ X\subset L^{2q_1}_TH^{\frac{s'+1}{2}} $: for any $ u\in X $,
	\[
		\begin{aligned}
			\|u\|_{L^{2q_1}_TH^{\frac{s'+1}{2}} } =\, & \left(\int_0^T\|u\|_{H^{\frac{s'+1}{2}}}^{2q_1}(t)dt\right)^{\frac{1}{2q_1}} \\
			\le\, &\left(\int_0^T\left(\|u\|^{\frac{s'+1}{2}}_{H^1}\|u\|_{L^2}^{\frac{1-s'}{2}}\right)^{2q_1}dt\right)^{\frac{1}{2q_1}} \\
			\le\, &\|u\|_{L^\infty_TL^2}^{\frac{1-s'}{2}}\|u\|_{L^2_TH^1}^{\frac{1}{q_1}}\\
			\le\, &\|u\|_X.
		\end{aligned}
	\]
	
	Substituting \eqref{v2b} $ \sim $ \eqref{vv1} back into \eqref{estTEMP} yields
	\begin{eqnarray*}
			\|v\|_X &\lesssim& T^{\frac{1}{q_1}}C_T\left(\|E\|_{L^\infty_TL^2}+\|v'\|_{L^\infty_TL^2}\right)\\
			&&+T^{\frac{1}{q_1}-\frac{1}{2}}\left(\|v^2\|_{X}\|B\|_{L^\infty_TL^2}+\|B^1\|_{L^\infty_TL^2}\|v\|_{X}\right)\\
			&&+C_T\left(\|v^2\|_{{}{L^{2q_1}_TH^{\frac{s'+1}{2}}}}\|v\|_X+\|v^1\|_{{}{L^{2q_1}_TH^{\frac{s'+1}{2}}}}\|v\|_X\right).
	\end{eqnarray*}
	}
Moreover, thanks to Lemma \ref{MW}, we have
\[
	\|E\|_{L^\infty_TL^2}+{}{\sqrt{\frac{1}{\varepsilon_0\mu_0}}}\|B\|_{L^\infty_TL^2} \lesssim \|v\|_{L^1_TL^2}+\|v'\|_{L^1_TL^2}\lesssim T(\|v\|_{X}+\|v'\|_{X}).
\]
Finally, choosing $ T $ small enough, we obtain
\[\begin{aligned}
	&\|v\|_{X}+\|v'\|_{X}+\|E\|_{L^\infty_TL^2}+{}{\sqrt{\frac{1}{\varepsilon_0\mu_0}}}\|B\|_{L^\infty_TL^2} \\
	\le\,& 1/2(\|v\|_{X}+\|v'\|_{X}+\|E\|_{L^\infty_TL^2}+{}{\sqrt{\frac{1}{\varepsilon_0\mu_0}}}\|B\|_{L^\infty_TL^2}),\end{aligned}
	\]
implying that $ \+=\-=0 $ and $ E=B=0 $ which yields the uniqueness of the solution on a small time interval. We can repeat this argument and get the global uniqueness. 
\section{Global Weak Solutions in 3D Case}
In this section we prove Theorem \ref{3dMain1}. The proof is similar to the incompressible Navier-Stokes equations. Again, we consider the frequency cut off system in previous section \eqref{approxi_ver2} and {}{use the Aubin-Lions-Simon Lemma (see for example \cite{SIMON} and \cite{Teman}) to prove that the solutions of cut off system converges to a weak solution of the original system}. We know \eqref{approxi_ver2} has a unique solution $ (v_-^k,v_+^k,E^k,B^k)\in C([0,\infty);L^2) $. According to the energy estimate, $ (v_-^k,v_+^k,E^k,B^k) $ is bounded in
\[
	L^\infty_TL^2 \cap L^2_T\dH^1 \times L^\infty_TL^2 \cap L^2_T\dH^1 \times L^\infty_T L^2 \times L^\infty_T L^2.
\]
It is sufficient to prove $ \partial_t v^k $ is bounded in $ L^2_T H^{-3/2} $. Then following the compactness argument in previous section finishes the proof.\\
By \eqref{approxi_ver2}, we obtain
\begin{eqnarray*}
	\|\partial_t v^k\|_{L^2_T H^{-3/2}} &\lesssim& \|\Delta v^k\|_{L^2_T H^{-3/2}}+\|E^k\|_{L^2_T H^{-3/2}}+\|v^k\|_{L^2_T H^{-3/2}}+\|v'^k\|_{L^2_T H^{-3/2}}\\
	&&{}{+}\|v^k\nabla v^k\|_{L^2_T H^{-3/2}}+\|v^k\times B^k\|_{L^2_T H^{-3/2}} \\
	&\lesssim& \|v^k\|_{L^2_T H^{1/2}} + T^{1/2}(\|E^k\|_{L^\infty_TL^2} +\|v^k\|_{L^\infty_TL^2} +\|v'^k\|_{L^\infty_TL^2} ) \\
	&&{}{+}\|v^k\nabla v^k\|_{L^2_T H^{-3/2}}+\|v^k\times B^k\|_{L^2_T H^{-3/2}}.
\end{eqnarray*}
By interpolation, 
\[
	\|v^k\|_{L^2_T H^{1/2}} \lesssim \|v^k\|_{L^2_TL^2}+\|v^k\|_{L^2_T\dH^1} \lesssim T^{1/2}\|v^k\|_{L^\infty_TL^2}+\|v^k\|_{L^2_T\dH^1}.
\]
Hence
\begin{equation*}
	\|\partial_t v^k\|_{L^2_T H^{-3/2}} \lesssim C_0+4T^{1/2}C_0+\|v^k\nabla v^k\|_{L^2_T H^{-3/2}}+\|v^k\times B^k\|_{L^2_T H^{-3/2}}.
\end{equation*}
Moreover, by Sobolev embedding and H\"older's inequality,
\begin{eqnarray*}
	\|v^k\nabla v^k\|_{L^2_T H^{-3/2}} &=& \|\div (v^k\otimes v^k)\|_{L^2_T H^{-3/2}} \\
	&\le&\|v^k\otimes v^k\|_{L^2_T H^{-1/2}} \\
	&\lesssim& \|v^k\otimes v^k\|_{L^2_TL^{3/2}} \\
	&\lesssim& \|v^k\|_{L^\infty_TL^2}\|v^k\|_{L^2_TL^6}\\
	&\lesssim&\|v^k\|_{L^\infty_TL^2}\|v^k\|_{L^2_T\dH^1}\lesssim C_0^2.
\end{eqnarray*}
With the help of \eqref{est1}, one get{}{s}
\begin{eqnarray*}
	\|v^k\times B^k\|_{L^2_T H^{-3/2}} &\le& \|v^k\times B^k\|_{L^2_T H^{s-3/2}} \\
	&\lesssim& \|v^k\|_{L^2_TH^s} \|B^k\|_{L^\infty_TL^2}\\
	&\lesssim& \|v^k\|_{L^2_TH^1} \|B^k\|_{L^\infty_TL^2} \\
	&\lesssim& (T^{1/2}\|v^k\|_{L^\infty_TL^2}+\|v^k\|_{L^2_T\dH^1}) \|B^k\|_{L^\infty_TL^2} \\
	&\lesssim&C_0^2(T^{1/2}+1).
\end{eqnarray*}
Therefore, we get the bound for $ \|\partial_t v\|_{L^2_TH^{-3/2}} $,
\[
	\|\partial_t v\|_{L^2_T H^{-3/2}} \lesssim C_0+4T^{1/2}C_0+C_0^2(T^{1/2}+2).
\]
This completes the proof.
\section{Local well-Posedness for 3D Case}
In this section, we prove Theorem \ref{3dMain2}. {The idea is to apply the {}{fixed} point argument.}\\
{\sc Step 1: A priori estimate.}\\
By Lemma \ref{NS}, for any $(v_-,v_+,E,B)$ solution of \eqref{physic_model}, we have
\begin{equation}\label{es1}
\|v\|_{L^\infty_T{H}^\frac{1}{2}\cap L^2_T{H}^\frac{3}{2}}\le C_T( \|v(0)\|_{{H}^\frac{1}{2}}+\|v\times B\|_{L^2_T{H}^{-\frac{1}{2}}}+\|R\|_{\tL^1_TH^{\frac{1}{2}}} + \|E+v\cdot\nabla v\|_{\tilde{L}^\frac{4}{3}_TL^2}),
\end{equation}
where $v$ stands for $v_-$ or $v_+$. 
And by lemma \ref{MW}, we obtain:
\begin{equation}\label{es3}
\|E\|_{L^\infty_TL^2}+\|B\|_{L^\infty_TL^2} \le {}{\gamma}(\|E(0)\|_{L^2}+\|B(0)\|_{L^2}+\|v_-\|_{L^1_TL^2}+\|v_+\|_{L^1_TL^2}),
\end{equation}
{}{where $ \gamma=\max(1,\sqrt{\frac{1}{\varepsilon_0\mu_0}})/{\min(1,\sqrt{\frac{1}{\varepsilon_0\mu_0}})} $.\\}
{\sc Step 2: Contraction argument.}\\
Let $\Gamma := ({}{\-},{}{\+},E,B)^T$ be such that
\begin{eqnarray*}
	\- &\in& X^{{}{\-}} := L^\infty_T {H}^{\frac{1}{2}} \cap L^2_T {H}^{\frac{3}{2}} \\
	\+ &\in& X^{{}{\+}} := L^\infty_T {H}^{\frac{1}{2}} \cap L^2_T {H}^{\frac{3}{2}} \\
	E &\in& X^E := L^\infty_T L^2 \\
	B &\in& X^B := L^\infty_T L^2,
\end{eqnarray*}
and set $X=X^\-\times X^\+\times X^E\times X^B$. Then the norm of $\Gamma$ can be defined as $\|\Gamma\|_X:=\|\-\|_{X^\-}+\|\+\|_{X^\+}+\|E\|_{X^E}+\|B\|_{X^B}$. We look for a solution in the following integral form
\[
\Gamma(t)=e^{tA}\Gamma(0) + \int^t_0 e^{(t-s)A}f(\Gamma(s))ds.
\]
The operator $A$ and function $f(\Gamma)$ are defined as
\[
A=\left(
\begin{array}{cccc}
{}{\frac{\nu_-}{nm_-}}\Delta & 0 & 0 & 0 \\
0 & {}{\frac{\nu_+}{nm_+}}\Delta & 0 & 0 \\
0 & 0 & 0 & \frac{1}{\varepsilon_0\mu_0}\nabla\times \\
0 & 0& -\nabla\times & 0
\end{array}
\right),
\]
\[
f(\Gamma) = \left(
\begin{array}{c}
\mathcal{P}(-v_-\cdot\nabla v_--\frac{e}{m_-}(E+v_-\times B)-\frac{R}{nm_-}) \\
\mathcal{P}(-v_+\cdot\nabla v_++\frac{eZ}{m_+}(E+v_{{}{+}}\times B)+\frac{R}{nm_+}) \\
-\frac{ne}{\varepsilon_0}(Zv_+-v_-)\\
0
\end{array}
\right).
\]
Define a map $\Phi: X \to X$ as
\[
\Phi(\Gamma):= \int^t_0 e^{(t-s)A}f(e^{sA}\Gamma_0 + \Gamma(s))ds
\]
where $\Gamma_0=(\-_0,\+_0,E_0,B_0)^T$. We denote the components 
$$
\Phi(\Gamma)=(\Phi(\Gamma)^\-,\Phi(\Gamma)^\+,\Phi(\Gamma)^E,\Phi(\Gamma)^B).
$$
Note that $\Phi(-e^{tA}\Gamma_0)=0$, and by Lemma \ref{NS}, 
\[
\|e^{tA}\Gamma_0\|_X \le C \|\Gamma_0\|_{{H}^{\frac{1}{2}} \times H^{\frac{1}{2}} \times L^2\times L^2},
\]
where $C$ is a universal constant. Moreover, setting $r=C\|\Gamma_0\|_{H^\frac{1}{2}\times H^\frac{1}{2}\times L ^2 \times L^2}$ and denoting by $B_r$ the ball centered at $0$ with radius $r$ in the space $ X $. Our goal is to {}{prove} if $T$ is small enough then $\Phi(B_r)\subset B_r$. \\
Assume $\Gamma\in B_r$ and set $\bar{\Gamma}:=e^{tA}\Gamma_0{}{+}\Gamma$ (also define $\bar{v}_-,\bar{v}_+,\bar{E},\bar{B},\bar{R}$ in the same manner). Then by \eqref{es1} and Lemma \ref{product}, 
\begin{eqnarray*}
	\|\Phi(\Gamma)\|_{X^{v}} &\le& C_T(\|\bar{v}\times \bar{B}+\bar{v}\cdot \nabla \bar{v}\|_{L^2_T{H}^{-\frac{1}{2}}}+\|\bar{E}\|_{\tilde{L}^\frac{4}{3}_TL^2}+\|\bar{R}\|_{{}{\tL}^1_T{H}^{\frac{1}{2}}})\\
	&\le&C_T\left(\|\bar{v}\|_{L^2_T{H}^1}+\|\bar{v}\|_{L^2_T{H}^\frac{3}{2}} + T^\frac{3}{4} + 2T\right)\|\bar{\Gamma}\|_X\\
	&\le&2C_T\left(\|\bar{v}\|_{L^2_T{H}^1}+\|\bar{v}\|_{L^2_T{H}^\frac{3}{2}} + T^\frac{3}{4} + 2T\right)r.
\end{eqnarray*}
Thus we could choose $T$ small such that
\[
2C_T\left(\|\bar{v}\|_{L^2_T{H}^1}+\|\bar{v}\|_{L^2_T{H}^\frac{3}{2}} + T^\frac{3}{4} + 2T\right)< \frac{1}{4}.
\]
So
\[
\|\Phi(\Gamma)\|_{X^{v_-}}<\frac{1}{4}r,\quad\|\Phi(\Gamma)\|_{X^{v_+}}<\frac{1}{4}r.
\]
Similarly, by \eqref{es3} we could also get
\begin{eqnarray*}
	&&\|\Phi(\Gamma)\|_{{}{X^E}}+\|\Phi(\Gamma)\|_{{}{X^B}}\\
	&\le&{}{\gamma}\|\bar{v}_-\|_{L^1_TL^2}+{}{\gamma}\|\bar{v}_+\|_{L^1_TL^2}\\
	&\le&{}{\gamma}T\|\bar{v}_-\|_{L^\infty_TH^{\frac{1}{2}}}+{}{\gamma}T\|\bar{v}_+\|_{L^\infty_TH^{\frac{1}{2}}}\\
	&\le&4{}{\gamma}Tr.
\end{eqnarray*}
By choosing $ T $ small enough, one obtains
\[
\|\Phi(\Gamma)\|_{{}{X^E}}+\|\Phi(\Gamma)\|_{{}{X^B}}<\frac{1}{2}r.  
\]
Thus
\[
\|\Phi(\Gamma)\|_X<r.
\]
And then $\Phi(\Gamma)\in B_r$.
Let $v$ be $v_-$ or $v_+$, $\Gamma_1, \Gamma_2\in B_r$ and $\bar{\Gamma}_i=e^{tA}\Gamma_0+\Gamma_i$, $i=1,2$. By the {\it a priori estimate} \eqref{es1}:
\begin{eqnarray}\label{lastES1}
\|\Phi(\Gamma_1)-\Phi(\Gamma_2)\|_{X^v} &\le& C_T\|(\bar{v}_1-\bar{v}_2)\times \bar{B}_2+\bar{v}_1\times(\bar{B}_1-\bar{B}_2)\|_{L^2_T{H}^{-\frac{1}{2}}} \nonumber\\
&& + C_T\|(\bar{v}_1-\bar{v}_2)\cdot\nabla \bar{v}_2+\bar{v}_1\cdot\nabla (\bar{v}_1-\bar{v}_2)\|_{\tilde{L}^\frac{4}{3}_TL^2} \\
&& {}{+}C_T\|\bar{E}_1-\bar{E}_2\|_{\tilde{L}^\frac{4}{3}_TL^2} + C_T\|\bar{R}_1-\bar{R}_2\|_{L^1_T{H}^{\frac{1}{2}}}. \nonumber
\end{eqnarray}
The fact that
\begin{eqnarray}
\|\bar{R}_1-\bar{R}_2\|_{L^1_T{H}^{\frac{1}{2}}} &\le& {}{\alpha}T\|v_{+,1}-v_{+,2}+v_{-,2}-v_{-,1}\|_{L^\infty_T{H}^\frac{1}{2}} \nonumber\\
\label{Est_R}&\le&2{}{\alpha}T\|\Gamma_1-\Gamma_2\|_X 
\end{eqnarray}
together with \eqref{PE2} and \eqref{PE4} in Lemma \ref{product}, \eqref{lastES1} becomes
\begin{eqnarray*}
	\|\Phi(\Gamma_1)-\Phi(\Gamma_2)\|_{X^v} &\le& C_TT^\frac{1}{4}\left(\|\bar{B}_2\|_{L^\infty_TL^2}+\|\bar{v}_1\|^\frac{1}{2}_{L^\infty_T{H}^\frac{1}{2}} \|\bar{v}_1\|^\frac{1}{2}_{L^2_T{H}^\frac{3}{2}}\right)\|\Gamma_1-\Gamma_2\|_X \\
	&&+C_T\left(\|\bar{v}_2\|_{L^2_T{H}^\frac{3}{2}}+\|\bar{v}_1\|^\frac{1}{2}_{L^\infty_T{H}^\frac{1}{2}} \|\bar{v}_1\|^\frac{1}{2}_{L^2_T{H}^\frac{3}{2}}\right)\|\Gamma_1-\Gamma_2\|_X \\
	&&+C_T(T^\frac{3}{4}+2{}{\alpha}T)\|\Gamma_1-\Gamma_2\|_X \\
	&\le& C_TT^\frac{1}{4}(4r)\|\Gamma_1-\Gamma_2\|_X \\
	&&+C_T\left(\|\bar{v}_2\|_{L^2_T{H}^\frac{3}{2}}+(2r)^\frac{1}{2} \|\bar{v}_1\|^\frac{1}{2}_{L^2_T{H}^\frac{3}{2}}\right)\|\Gamma_1-\Gamma_2\|_X \\
	&&+C_T(T^\frac{3}{4}+2{}{\alpha}T)\|\Gamma_1-\Gamma_2\|_X.
\end{eqnarray*}
We choose $T$ small enough so that
\[
\|\Phi(\Gamma_1)-\Phi(\Gamma_2)\|_{X^v} \le \frac{1}{8}\|\Gamma_1-\Gamma_2\|_X.
\]
Similarly, {}{by \eqref{es3}} we can estimate $\|\Phi(\Gamma_1)-\Phi(\Gamma_2)\|_{X^E}+\|\Phi(\Gamma_1)-\Phi(\Gamma_2)\|_{X^B}$ when $T$ is small:{}{ 
\begin{eqnarray*}
	&&\|\Phi(\Gamma_1)-\Phi(\Gamma_2)\|_{X^E}+\|\Phi(\Gamma_1)-\Phi(\Gamma_2)\|_{X^B} \\
	&\le&{}{\gamma}C_T(\|\bar{v}_{-,1}-\bar{v}_{-,2}\|_{L^1_TL^2}+\|\bar{v}_{+,1}-\bar{v}_{+,2}\|_{L^1_TL^2})\\	
	&\le& {}{\gamma}C_TT(\|\bar{v}_{-,1}-\bar{v}_{-,2}\|_{L^\infty_TH^\frac{1}{2}}+\|\bar{v}_{+,1}-\bar{v}_{+,2}\|_{L^\infty_TH^\frac{1}{2}}) \\
	&\le&\frac{1}{4}\|\Gamma_1-\Gamma_2\|_X,
\end{eqnarray*}
where $ v=v_\pm. $\\}
Finally, together with the above estimates
\[
\|\Phi(\Gamma_1)-\Phi(\Gamma_2)\|_X \le \frac{1}{2}\|\Gamma_1-\Gamma_2\|_X.
\]
Thus a fixed point argument gets the local existence.\\
{\sc Step 3: Global existence for small initial data.}\\
In this step, the universal constant will be written explicitly in order to see how the physical parameters affect  {}{the estimates and how the solutions depend upon them. This will be needed in a forthcoming work about relaxation limits}. We rewrite $ X^v_T=L^\infty_T{H}^\frac{1}{2}\cap L^2_T{H}^\frac{3}{2} $ to emphasize the dependence upon time $ T $ and recall the system \eqref{physic_model_1},
\begin{equation*}
\begin{cases}
\partial_t v_--\mu_-\Delta\-+v_-\cdot \nabla v_-= - a_-(E+v_-\times B) + b_-(\+-\-)- \frac{1}{nm_-}\nabla p_- \\
\partial_t v_+-\mu_+\Delta\++v_+\cdot \nabla v_+=  a_+(E+v_+\times B) - b_+(\+-\-)- \frac{1}{nm_+}\nabla p_+ \\
\partial_t E = \frac{1}{\varepsilon_0\mu_0}\nabla\times B-\frac{ne}{\varepsilon_0}(Zv_+-v_-) \\
\partial_t B = -\nabla\times E \\
\text{div} v_-=\text{div} v_+=\text{div} B=\text{div} E=0,\\
\end{cases}
\end{equation*}
where $ \mu_\pm=\frac{\nu_\pm}{nm_\pm} $, $ a_+=\frac{eZ}{m_+} $, $ a_-=\frac{e}{m_-} $, $ b_\pm=\frac{\alpha}{nm_\pm} $.\\
After setting
\[
\lambda_1=\max\{\frac{nm_-}{2\varepsilon_0}, \frac{nm_+}{2\varepsilon_0}, \frac{1}{2}, \frac{1}{2\varepsilon_0\mu_0}\},
\]
and
\[
\lambda_2=\min\{\frac{nm_-}{2\varepsilon_0}, \frac{nm_+}{2\varepsilon_0}, \frac{1}{2}, \frac{1}{2\varepsilon_0\mu_0}, \frac{\nu_-}{\varepsilon_0}, \frac{\nu_+}{\varepsilon_0}, \frac{\alpha}{\varepsilon_0}\},
\]
Energy estimate given by Lemma \ref{energyEst} gives us the bounds
\begin{equation}\label{C0Bound}
\|v_\pm\|_{L^1_T\dH^1},\|v_\pm\|_{L^\infty_TL^2}, \|E\|_{L^\infty_TL^2} \le \sqrt{{}{\frac{\lambda_1}{\lambda_2}}}B_0{}{,}
\end{equation}
{}{where $ B_0=\|v_-(0)\|_{L^2}+\|v_+(0)\|_{L^2}+\|E(0)\|_{L^2}+\|B(0)\|_{L^2} $.\\
Next we do energy estimate in $ \dH^{\frac{1}{2}} $, and prove that if $ \|v_\pm(0)\|_{\dH^{\frac{1}{2}}} $ is small enough, the $ \|v_\pm(t)\|_{\dH^{\frac{1}{2}}} $ remains small after some time and $ \|v_\pm\|_{X_T^{v_\pm}} $ remains bounded so that one can extend the time of existence to infinity. Multiplying the first and second equations of \eqref{physic_model_1} by $ |\nabla|v_-  $ and $ |\nabla|v_+ $ respectively and integrating over $ \mathbb{R}^3 $ one obtains
\begin{eqnarray}\label{H1/2energy_temp}
	\frac{1}{2}\frac{d}{dt}\|v_\pm\|_{\dH^{1/2}}^2+\mu_\pm\|v_\pm\|^2_{\dH^{3/2}}&\le&(a_\pm\|E\|_{L^2}+b_\pm\|v_+-v_-\|)\|v_\pm\|_{\dH^{1}} \nonumber\\
	&&+{}{a_\pm}{}{c_1}\|B\|_{L^2}\|v_\pm\|_{\dH^1}\|v_\pm\|_{\dH^{3/2}}\nonumber\\
	&&+{}{c_1}\|v_\pm\|_{\dH^1}^2\|v_\pm\|_{\dH^{3/2}}\\
	&\le&(a_\pm+2b_\pm)\sqrt{\lambda_1/\lambda_2}B_0\|v_\pm\|_{\dH^{1}} \nonumber\\
	&&+a_\pm {}{c_1}\sqrt{\lambda_1/\lambda_2}B_0\|v_\pm\|_{\dH^1}\|v_\pm\|_{\dH^{3/2}}\nonumber\\
	&&+{}{c_1}\|v_\pm\|_{\dH^1}^2\|v_\pm\|_{\dH^{3/2}},\nonumber
\end{eqnarray}
where we use the following inequality for nonlinear terms,
\[
	|(a\times b, |\nabla|c)|\le \|b\|_{L^2}\|a\|_{L^6}\||\nabla|c\|_{L^3}\le {}{c_1}\|b\|_{L^2}\|a\|_{\dH^1}\|c\|_{\dH^{3/2}},
\]
and
\[
	|(a\nabla b, |\nabla|c)| \le \|a\|_{L^6}\|\nabla b\|_{L^3}\||\nabla|c\|_{L^2}\le {}{c_1}\|a\|_{\dH^1}\|c\|_{\dH^1}\|b\|_{\dH^{3/2}},
\]
and $ c_1 $ here is a constant only depending on dimension.
Letting $ \mu=\min(\mu_-,\mu_+) $ and $ v=v_\pm $ we rewrite \eqref{H1/2energy_temp} by
\begin{equation}\label{H1/2energy1}
	\begin{aligned}
		\frac{1}{2}\frac{d}{dt}\|v\|_{\dH^{1/2}}^2+\mu\|v\|^2_{\dH^{3/2}} \le\, &(a_\pm+2b_\pm)\sqrt{\frac{\lambda_1}{\lambda_2}}B_0\|v\|_{\dH^{1}}+a_\pm{}{c_1}\sqrt{\frac{\lambda_1}{\lambda_2}}B_0\|v\|_{\dH^1}\|v\|_{\dH^{3/2}}\\ &+{}{c_1}\|v\|_{\dH^1}^2\|v\|_{\dH^{3/2}}.\end{aligned}
\end{equation}
Assuming that $ \|v_\pm\|_{\dH^{1/2}}(0)\le A_0 < \frac{\mu}{2{}{c_1}} $, by continuity, there exists time $ T^*(A_0) $ such that for all $ 0\le t\le T^*(A_0) $, $ \|v_\pm\|_{\dH^{1/2}}(t)\le \frac{\mu}{2c_1} $. We consider \eqref{H1/2energy1} for $ t\le T^*(A_0) $. After using the interpolation  $ \|v\|_{\dH^{1}}\le\|v\|_{\dH^{1/2}}^{1/2}\|v\|_{\dH^{3/2}}^{1/2} $ for every terms on the right hand side one can get rid of the last term in \eqref{H1/2energy1} and obtains
\begin{equation}\label{H1/2energy2}
		\begin{aligned}
			\frac{1}{2}\frac{d}{dt}\|v\|_{\dH^{1/2}}^2+\frac{\mu}{2}\|v\|^2_{\dH^{3/2}} \le\,& (a_\pm+2b_\pm)\sqrt{\frac{\lambda_1}{\lambda_2}}B_0\|v\|_{\dH^{1/2}}^{1/2}\|v\|^{1/2}_{\dH^{3/2}}\\
			&+a_\pm{}{c_1}\sqrt{\frac{\lambda_1}{\lambda_2}}B_0\|v\|_{\dH^{1/2}}^{1/2}\|v\|^{3/2}_{\dH^{3/2}}.\end{aligned}
\end{equation}
By Young's inequality, there exist{}{s} a universal constant $ c $ such that
\begin{eqnarray*}
		\frac{1}{2}\frac{d}{dt}\|v\|_{\dH^{1/2}}^2+\frac{\mu}{2}\|v\|^2_{\dH^{3/2}} &\le&c\mu^{-1/3} \left((a_\pm+2b_\pm)\sqrt{\frac{\lambda_1}{\lambda_2}}B_0\right)^{4/3}\|v\|^{2/3}_{\dH^{1/2}}\\
		&&+c\mu^{-3}\left(a_\pm{}{c_1}\sqrt{\frac{\lambda_1}{\lambda_2}}B_0    \right)^4\|v\|^2_{\dH^{1/2}}+\frac{\mu}{4}\|v\|_{\dH^{3/2}}^2.
\end{eqnarray*}
Since $ \|v\|_{\dH^{1/2}}<\mu/(2{}{c_1}) $ for $ t\le T^*(A_0) $, we have
\[\begin{aligned}
	\frac{d}{dt}\|v\|_{\dH^{1/2}}^2\le\, &2c\mu^{-1/3} \left((a_\pm+2b_\pm)\sqrt{\frac{\lambda_1}{\lambda_2}}B_0\right)^{4/3}\left(\frac{\mu}{2{}{c_1}}\right)^{2/3}\\&
	+2c\mu^{-3}\left(a_\pm{}{c_1}\sqrt{\frac{\lambda_1}{\lambda_2}}B_0\right)^4\left(\frac{\mu}{2{}{c_1}}\right)^2.\end{aligned}
\]
Let $ C=c^{1/2}\max(\mu^{1/6}\left((a_\pm+2b_\pm)\sqrt{\frac{\lambda_1}{\lambda_2}}B_0\right)^{2/3}\left(\frac{1}{2{}{c_1}}\right)^{1/3}, \mu^{-1/2}\left(a_\pm{}{c_1}\sqrt{\frac{\lambda_1}{\lambda_2}}B_0\right)^2\left(\frac{1}{2{}{c_1}}\right)) $, we end up with
\[
	\|v\|_{C([0,t];\dH^{1/2})}\le A_0+2Ct.
\]
Therefore $ T^*(A_0)\ge\frac{\frac{\mu}{2{}{c_1}}-A_0}{2C} $.\\
To prove $ \|v\|_{X_T^v} $ is bounded, we go back to \eqref{H1/2energy1}. For the right hand side of \eqref{H1/2energy1}, apply $ \dH^1 $ interpolation only for the last term to get rid of it and Young's inequality for the first and second terms,
\begin{eqnarray*}
		\frac{1}{2}\frac{d}{dt}\|v\|_{\dH^{1/2}}^2+\frac{\mu}{2}\|v\|^2_{\dH^{3/2}} &\le& \frac{1}{2}\frac{\lambda_1}{\lambda_2}(a_\pm+2b_\pm)^2 B_0^2+\frac{1}{2}\|v\|^2_{\dH^{1}} \\
		&&+C_\mu \frac{\lambda_1}{\lambda_2} a_\pm^2{}{c_1^2}B_0^2\|v\|^2_{\dH^1}+\frac{\mu}{4}\|v\|^2_{\dH^{3/2}}.
\end{eqnarray*}
Noticing that $ \|v\|_{\dH^{3/2}}^2(t)=\frac{d}{dt}\|v\|^2_{L^2_t\dH^{3/2}} $, we have
\[
	\frac{d}{dt}\left(\frac{1}{2}\|v\|_{\dH^{1/2}}^2+\frac{\mu}{4}\|v\|^2_{L^2_t\dH^{3/2}}   \right) \le \frac{1}{2}\frac{\lambda_1}{\lambda_2}(a_\pm+2b_\pm)^2 B_0^2+(\frac{1}{2}+c\mu^{-1} \frac{\lambda_1}{\lambda_2} a_\pm^2{}{c_1^2}B_0^2)\|v\|^2_{\dH^{1}}.
\]
Therefore for $T< T^*(A_0)$, 
\[
	\frac{1}{2}\|v\|_{C([0,T];\dH^{1/2})}^2+\frac{\mu}{4}\|v\|^2_{L^2_T\dH^{3/2}} \le \frac{1}{2}A_0^2+\frac{1}{2}\frac{\lambda_1}{\lambda_2}(a_\pm+2b_\pm)^2 B_0^2T+(\frac{1}{2}+c\mu^{-1} \frac{\lambda_1}{\lambda_2} a_\pm^2{}{c_1^2}B_0^2)B_0^2T,
\]
which yields the boundedness of $ \|v\|_{X^v_T} $ since $ \|v(t)\|_{L^2} $ is always bounded by energy estimate. Therefore according to the local existence proof at Step 2, the solution exists up to the time T and $  \frac{\frac{\mu}{2{}{c_1}}-A_0}{2C}\le T<T^*(A_0) $. Furthermore, $ \|v(t)\|_{\dH^{\frac{1}{2}}}<\frac{\mu}{2{}{c_1}} $ for all $ t<T $. One can repeat the extension argument above starting at time $ T $ and hence the time of existence can be extended to infinity.
}

\vspace{0.5cm}

\noindent
{\bf Acknowledgments.}\
The authors thank the anonymous referees for their thorough review and suggestions of the first draft of this paper.\\
SI and SS were partially supported by NSERC grant (371637-2014).
YG was partially supported by the Japan Society for the 
Promotion of Science (JSPS) through the grants Kiban S (26220702), {}{Kiban B (16H03948)}, 
Kiban A (23244015) and Houga (2560025), and also partially supported by the German
Research Foundation through the Japanese-German Graduate Externship
and IRTG 1529. TY was partially supported by 
Grant-in-Aid for Young Scientists B (25870004),
Japan Society for the Promotion of Science (JSPS).
SI wants also to thank  Department of Mathematics at  Hokkaido University for their great hospitality and support during his visit.
Also, this paper was
developed during a stay of TY as an assistant professor
of the Department of Mathematics at Hokkaido University and as an associate professor of the
Department of Mathematics at Tokyo Institute of Technology.

\end{document}